\documentclass{elsarticle}
\usepackage[english]{babel}
\usepackage[utf8]{inputenc}
\usepackage[T1]{fontenc}
\usepackage{amsmath}
\usepackage{amsthm}
\usepackage{amssymb}
\usepackage{amsfonts}
\usepackage{fullpage}
\usepackage{cleveref}
\usepackage{microtype}
\usepackage{tabu}
\usepackage{diagbox}
\usepackage{multirow}
\usepackage{makecell}
\usepackage{todonotes}

\usepackage{tikz}
\usetikzlibrary{calc}
\tikzstyle{noeud}=[circle,inner sep=2, minimum size =3 pt, line width = 1pt, draw=black, fill=white]
\definecolor{bleu}{rgb}{0.38, 0.31, 0.86}
\definecolor{rouge}{RGB}{162, 27, 43}

\makeatletter

\newskip\@bigflushglue \@bigflushglue = -100pt plus 1fil

\def\bigcentering{\let\\\@centercr\rightskip\@bigflushglue%
\leftskip\@bigflushglue
\parindent\z@\parfillskip\z@skip}

\makeatother

\usepackage{zref-savepos}
\newcounter{NoTableEntry}
\renewcommand*{\theNoTableEntry}{NTE-\the\value{NoTableEntry}}
\newcommand*{\notableentry}{%
	\multicolumn{1}{@{}c@{}|}{%
		\stepcounter{NoTableEntry}%
		\vadjust pre{\zsavepos{\theNoTableEntry t}}
		\vadjust{\zsavepos{\theNoTableEntry b}}
		\zsavepos{\theNoTableEntry l}
		\hspace{0pt plus 1filll}%
		\zsavepos{\theNoTableEntry r}
		\tikz[overlay]{%
			\draw[black]
			let
			\n{llx}={\zposx{\theNoTableEntry l}sp-\zposx{\theNoTableEntry r}sp},
			\n{urx}={0},
			\n{lly}={\zposy{\theNoTableEntry b}sp-\zposy{\theNoTableEntry r}sp},
			\n{ury}={\zposy{\theNoTableEntry t}sp-\zposy{\theNoTableEntry r}sp}
			in
			(\n{llx}, \n{lly}) -- (\n{urx}, \n{ury})
			(\n{llx}, \n{ury}) -- (\n{urx}, \n{lly})
			;
		}%
	}%
}

\DeclareMathOperator{\bal}{bal}
\DeclareMathOperator{\ci}{Ci}

\newtheorem{theorem}{Theorem}
\newtheorem{lemma}[theorem]{Lemma}
\newtheorem{definition}[theorem]{Definition}
\newtheorem{prop}[theorem]{Proposition}
\newtheorem{corollary}[theorem]{Corollary}
\newtheorem{conjecture}[theorem]{Conjecture}

\crefname{proposition}{proposition}{propositions}

\title{On the balanceability of some graph classes\tnoteref{supports}}
\tnotetext[supports]{This work has been supported by PAPIIT IN111819 and CONACyT project 282280.}
\author[imatem]{Antoine Dailly\corref{cor}}
\author[imatem]{Adriana Hansberg}
\author[imatem]{Denae Ventura}
\cortext[cor]{Corresponding author}
\address[imatem]{Instituto de Matem\'aticas, UNAM Juriquilla, 76230 Quer\'etaro, Mexico.}
\date{}

\begin{document}
	
	\begin{frontmatter}
		
		\begin{abstract}
			Given a graph $G$, a 2-coloring of the edges of $K_n$ is said to contain a \emph{balanced copy} of $G$ if we can find a copy of $G$ such that half of its edges are in each color class. If, for every sufficiently large $n$, there exists an integer $k$ such that every 2-coloring of $K_n$ with more than $k$ edges in each color class contains a balanced copy of $G$, then we say that $G$ is \emph{balanceable}.
			Balanceability was introduced by Caro, Hansberg and Montejano, who also gave a structural characterization of balanceable graphs.
			
			In this paper, we extend the study of balanceability by finding new sufficient conditions for a graph to be balanceable or not. We use those conditions to fully characterize the balanceability of graph classes such as rectangular and triangular grids, as well as a special class of circulant graphs.
		\end{abstract}
	
		
	\end{frontmatter}
	
	\section{Introduction}
	
	
	Ramsey Theory studies the presence of ordered substructures in large, arbitrarily ordered structures. For instance, the seminal Ramsey Theorem~\cite{R30} states that, for every integer $r$, every 2-coloring of the edges of $K_n$ contains a monochromatic $K_r$ whenever $n$ is sufficiently large. However, it is also possible to look for other kinds of ordered substructures. In particular, Caro, Hansberg and Montejano~\cite{CHM19, CHM20} introduced the notion of \emph{balanceability}, which looks for balanced copies of a graph in 2-colorings of the edges of $K_n$.
	
	
	More formally, let $G(V,E)$ be a simple, finite graph. A 2-coloring of the edges of $K_n$ is a function $\alpha:E(K_n) \rightarrow \{R,B\}$ that associates every edge with one of two colors $R$ and $B$, being called the color classes. A 2-coloring $\alpha$ of the edges of $K_n$ is said to contain a \emph{balanced copy} of $G$ if we can find a copy of $G$ such that its edge-set $E$ is partitioned in two parts $(E_1,E_2)$ such that $||E_1|-|E_2|| \le 1$ and $\alpha(e)=R$ for $e \in E_1$, and $\alpha(e)=B$ for $e \in E_2$. Said otherwise, we can find a copy of $G$ with half ($\pm 1$) of its edges in each color class. Being inspired by Ramsey Theory, balanceability is about finding a balanced copy of $G$ in any 2-coloring of the edges of $K_n$. However, it is trivial to see that, for any graph $G(V,E)$, we need at least $\frac{|E|}{2}$ edges in each color class in order to find a balanced copy of $G$. Thus, we need to guarantee a certain number of edges in each color class, leading to the following definition of a balanceable graph:
	
	\begin{definition}
		\label{def-balanceable}
		Let $G$ be a graph and $n$ be a positive integer. Let $\bal(n,G)$ be the smallest integer, if it exists, such that every 2-coloring $\alpha:E(K_n) \rightarrow \{R,B\}$ of the edges of the complete graph $K_n$ with $|{\alpha}^{-1}(R)|,|\alpha^{-1}(B)|>\bal(n,G)$ contains a balanced copy of $G$.
		
		If there is an $n_0$ such that, for every $n \geq n_0$, $\bal(n,G)$ exists, then $G$ is \emph{balanceable}.
	\end{definition}
	
	
	For example, a path on two edges is balanceable, since as long as each color class contains at least one edge, then we will be able to find two incident edges belonging to different color classes. For larger graphs, the question of deciding whether they are balanceable seems difficult to solve. Given a graph $G$ and a partition $(X,Y)$ of its vertices, we denote by $e(X,Y)$ the number of edges with one endpoint in $X$ and the other in $Y$; given a graph $G$ and a subset $W$ of its vertices, we denote by $e(G[W])$ the number of edges in the subgraph of $G$ induced by the vertices in $W$. In their paper, Caro, Hansberg and Montejano proved the following characterization of balanceable graphs:
	
	\begin{theorem}[\cite{CHM20}]
		\label{thm-characterization}
A graph $G(V,E)$ is balanceable if and only if $G$ has both a partition $V=X \cup Y$ and a set of vertices $W \subseteq V$ such that $e(X,Y),e(G[W]) \in \left\{\left\lfloor\frac{|E|}{2}\right\rfloor,\left\lceil\frac{|E|}{2}\right\rceil\right\}$.
	\end{theorem}
	
	In other words, a graph is balanceable if and only if it has both a cut crossed by exactly half of its edges (floor or ceiling) and an induced subgraph containing exactly half of its edges (floor or ceiling). \Cref{thm-characterization} was proved by showing that, for every integer $t$ and any $n$ sufficiently large, there exists a number $m(n,t)$ such that every 2-coloring of the edges of $K_n$ with more than $m(n,t)$ edges in each color class contains one of two specific colored copies of $K_{2t}$ (called \emph{universal unavoidable patterns} in \cite{CHM20}). Those specific colored copies may then be used, for some graph $G$, to find a balanced copy of $G$ or to prove that no balanced copy of $G$ may exist.
	
	Caro, Hansberg and Montejano~\cite{CHM20} showed via \Cref{thm-characterization} that trees are balanceable. It was shown in~\cite{CHM19}, previous to this result, that the only balanceable complete graph with an even number of edges is $K_4$. Amoebas, a wide family of balanceable graphs with interesting interpolation properties were introduced in \cite{CHM20} and further developed in \cite{CHM20-2}. Caro, Lauri and Zarb~\cite{CLZ19} exhaustively studied the balanceability of graphs of at most four edges. For other references related to balanceability of graphs, see also \cite{BHMM, CHLZ, CHM19-2}. However, the question of balanceability remains open for many graph classes.
	An explanation for this is that the characterization of \Cref{thm-characterization} can be unpractical to handle. In this paper, we will give weaker but more practical conditions for deciding whether a given graph is balanceable or not (\Cref{sec-conditions}), that we will apply to several graph classes, namely a special class of circulant graphs in \Cref{sec-circulant}, and rectangular and triangular grids in \Cref{sec-grids}.

	\section{Sufficient conditions for balanceable and non-balanceable graphs}
	\label{sec-conditions}
	
	In this section, we are concerned with finding weaker but more practical conditions for deciding whether a given graph can be balanceable. Those conditions will be based on the characterization of balanceable graphs provided by \Cref{thm-characterization}.
	We will first give conditions based on subsets of vertices verifying certain constraints, before studying more global properties and what they imply on the balanceability of the graph.
	
	\subsection{Conditions based on subsets of vertices verifying certain constraints}
	
	A first example is the following sufficient condition for a graph to be balanceable:
	
	\begin{prop}
		\label{prop-degreeHalfEdges}
		Let $G(V,E)$ be a graph. If there exists an independent set of vertices  $I \subset V$ such that $\sum_{v \in I} d(v) = \frac{|E|}{2}$, then $G$ is balanceable.
	\end{prop}
	
	\begin{proof}
		Let $X=I$ and $Y=W=V \setminus I$. Due to the condition on $I$, we have $e(X,Y)=\frac{|E|}{2}=e(G[W])$, and thus, by \Cref{thm-characterization}, $G$ is balanceable.
	\end{proof}
	
	This condition gives a direct proof of the balanceability of cycles of length equal to a multiple of~4:
	
	\begin{corollary}
		\label{cor-c4l}
		Let $\ell$ be a positive integer. The cycle $C_{4\ell}$ is balanceable.
	\end{corollary}
	
	\begin{proof}
		We denote the vertices of $C_{4\ell}$ by $u_0,u_1,\ldots,u_{4\ell-1}$. By setting $I = \{u_{4i}~|~i \in \{0,\ldots,\ell-1\}\}$, we can apply \Cref{prop-degreeHalfEdges} and get the result.
	\end{proof}
	
	Furthermore, we can set $I$ to be a singleton, in which case we obtain the following:
	
	\begin{prop}
		\label{prop-graphWithBigVertex}
		Let $G(V,E)$ be a graph. If there is a vertex $v \in V$ with $d(v)=\frac{|E|}{2}$, then $G$ is balanceable.
	\end{prop}
	
	\begin{proof}
		Let $I$ be the singleton containing the vertex $v$ with $d(v)=\frac{|E|}{2}$ and apply \Cref{prop-degreeHalfEdges}.
	\end{proof}
	
	This condition can be applied to the family of wheels:
	
	\begin{corollary}
		\label{cor-wheels}
		Wheels are balanceable.
	\end{corollary}
	
	\begin{proof}
		The wheel $W_n$ contains $2n$ edges, and the center of $W_n$ has degree $n$, hence \Cref{prop-graphWithBigVertex} applies.
	\end{proof}

	\subsection{Conditions based on global properties of the graph}
	
	We now focus on more global properties, such as regularity or whether the graph is eulerian or not.
	First, we have a condition guaranteeing that a graph cannot be balanced. Recall that a graph is \emph{eulerian} if all its vertices have even degree.
	
	\begin{prop}
		\label{prop-halfEdgesOddAndEulerian}
		Let $G(V,E)$ be an eulerian graph with an even number of edges. If $\frac{|E|}{2}$ is odd, then $G$ is not balanceable.
	\end{prop}
	
	\begin{proof}
		Assume by contradiction that $G$ is balanceable. Then, using \Cref{thm-characterization}, there is a partition of $V$ in two sets $X$ and $Y=V \setminus X$ such that half the edges are between $X$ and $Y$. Let us denote by $X_{odd}$ the set of vertices of $X$ that have an odd number of neighbours in $Y$. The fact that $\frac{|E|}{2}$ is odd implies that $|X_{odd}|$ is odd. This in turn implies that the vertices of $X_{odd}$ have odd degree in $G[X]$, since they have even degree in~$G$. Thus, $G[X]$ has an odd number of vertices with odd degree, which is impossible. This contradiction yields the result.
	\end{proof}
	
	This allows us to prove that cycles on $4\ell+2$ vertices are not balanceable:
	
	\begin{corollary}
		\label{cor-c4l+2}
		Let $\ell$ be a positive integer. The cycle $C_{4\ell+2}$ is not balanceable.
	\end{corollary}
	
	\begin{proof}
		The cycle $C_{4\ell+2}$ is eulerian, and has $4\ell+2$ edges, so $\frac{|E(C_{4\ell+2})|}{2}=2\ell+1$ is odd and we can apply \Cref{prop-halfEdgesOddAndEulerian}.
	\end{proof}
	
	We also study the balanceability of some specific regular graphs. First, \Cref{prop-halfEdgesOddAndEulerian} allows us to prove that some regular graphs are not balanceable:
	
	\begin{corollary}
		\label{cor-regular}
		Let $d$ be an even positive integer, and let $G$ be a $d$-regular graph of order $n$ with an even number of edges. $G$ is not balanceable in the following cases:
		\begin{enumerate}
			\item If $d,n \equiv 2 \bmod 4$;
			\item If $d = 4a$ with $a$ odd, and $n \equiv 1,3 \bmod 4$.
		\end{enumerate}
	\end{corollary}

	Note that this is a sufficient condition for non-balanceability, so a regular graph that verifies neither conditions may still be non-balanceable.
	
	\begin{proof}
		If $G$ is $d$-regular of order $n$ and size $e$, then $e=\frac{dn}{2}$. Furthermore, since $d$ is even, $G$ is eulerian.
		We will study the two cases:
		\begin{enumerate}
			\item Given $d=4a+2$ and $n=4b+2$, then $\frac{e}{2}=\frac{(4a+2)(4b+2)}{4}=4ab+2a+2b+1$ is odd. Hence, \Cref{prop-halfEdgesOddAndEulerian} implies that $G$ is not balanceable.
			\item Given $d=4a$ and $n=4b+c$ with $a$ odd and $c \in \{1,3\}$, then, $\frac{e}{2}=\frac{4a(4b+c)}{4}=4ab+ac$. Since $c \in \{1,3\}$, this has the same parity than $a$, hence it is odd and \Cref{prop-halfEdgesOddAndEulerian} implies that $G$ is not balanceable.
		\end{enumerate}
	\end{proof}
	
	Note that \Cref{cor-regular} easily implies that the complete graphs $K_{8k+5}$ are not balanceable, which was included in a more complete statement on the balanceability of complete graphs~\cite{CHM19}, but for which the proof was more complex.
	
	Finally, we state the following property about the balanceability of some bipartite regular graphs:
	
	\begin{prop}
		\label{prop-regularBipartite4n}
		Let $d$ be a positive integer. If $G$ is a bipartite, $d$-regular graph on $4n$ vertices, then $G$ is balanceable.
	\end{prop}
	
	\begin{proof}
		Let $G$ be a bipartite, $d$-regular graph of order $4n$, and let $A \cup B$ be its bipartition. Since $G$ is regular, $|A|=|B|=2n$, and $|E(G)|=2nd$.
		
		Let $I$ be a subset of $n$ vertices in $A$. Clearly, $I$ is an independent set, and by regularity $\sum_{v \in I} d(v) = nd = \frac{|E(G)|}{2}$. Thus, by \Cref{prop-degreeHalfEdges}, $G$ is balanceable.
	\end{proof}
	
	
%
	\ \\
	
	In the next two sections, we will characterize the balanceability of several graph classes, starting with a special class of circulant graphs.

	\section{Balanceability of a special class of circulant graphs}
	\label{sec-circulant}
	
	Circulant graphs are defined as the Cayley graphs of cyclic groups. Another definition of this family is the following. Let $k$ be a positive integer and $L$ be a list of integers in $\{1,\ldots,\lfloor\frac{k}{2}\rfloor\}$. The circulant graph $\ci_n(L)$ is the graph of order $k$ where the vertices are numbered from $0$ to $k-1$ and where, for every $j \in L$, the vertices $i$ and $(i+j) \bmod k$ are adjacent.
	
	Circulant graphs include empty graphs ($L=\emptyset$), cycles ($\ci_k(1)$), complete graphs ($\ci_k(1,2,\ldots,\lfloor\frac{k}{2}\rfloor)$), balanced complete bipartite graphs ($\ci_{2k}(1,3,\ldots,k)$ for $k$ odd and $\ci_{2k}(1,3,\ldots,k-1)$ for $k$ even), antiprism graphs ($\ci_{2k}(1,2)$), odd prism graphs ($\ci_{2k}(2,k)$ for $k$ odd), M\"obius ladders ($\ci_{2k}(1,k)$), as well as many other graph families.
	
	The balanceability of some families of circulant graphs has already been indirectly discussed, such as cycles (\Cref{cor-c4l,cor-c4l+2}) or complete graphs~\cite{CHM19}. In this section, we will focus on a specific class of circulant graphs, which can be defined as $\ci_k(1,\ell)$ for every integer $\ell$ in $\{2,\ldots,\lfloor\frac{k}{2}\rfloor\}$. Another way of defining this family is the following: let $k$ and $\ell$ be two integers such that $k>3$ and $\ell \in \{2,\ldots,k-2\}$, the graph $C_{k,\ell}$ is the cycle graph $C_k$ with vertices $u_0,\ldots,u_{k-1}$, and such that the chords $u_iu_{i+\ell}$ (where the addition is modulo $k$) are added to the edge set. Note that this family includes, among others, the antiprism graphs and M\"obius ladders.
	
	The following statement fully characterizes which of the graphs of the form $C_{k,\ell}$ are balanceable:
	
	\begin{theorem}
		\label{thm-circulant}
		Let $k$ and $\ell$ be two integers such that $k>3$ and $\ell \in \{2,\ldots,k-2\}$. The graph $C_{k,\ell}$ is balanceable if and only if $k$ is even and $(k,\ell) \neq (6,2)$.
	\end{theorem}

	Due to the many cases we have to consider, the proof of \Cref{thm-circulant} will be divided into several lemmas. The statement of \Cref{thm-circulant} and the different cases and lemmas that prove them are summarized in \Cref{tab-balCkl}. For the remainder of this section, we will assume that $\ell \leq \frac{k}{2}$, since if this is not the case then we can apply the same reasoning with $\ell' = k-\ell$.

%
%
%
%
	
	\begin{table}[h]
		\centering
		\bgroup
		\def\arraystretch{1.3}
		\begin{tabu}{|[1pt]c|[1pt]c|c|c|c|[1pt]}
			\tabucline[1pt]{-}
			\multirow{2}{*}{\diagbox[width=2cm,innerleftsep=.5cm,innerrightsep=.5cm]{$k$}{$\ell$}} & \multirow{2}{*}{$\frac{k}{2}$} & \multicolumn{3}{c|[1pt]}{$< \frac{k}{2}$} \\ \cline{3-5}
			
			& & odd & 2 & even, $>2$ \\ \tabucline[1pt]{-}
			
			odd & \notableentry & \multicolumn{3}{c|[1pt]}{Not balanceable (\Cref{lem-circulant-odd-l})} \\ \tabucline[1pt]{-}
			
			$\equiv 0 \bmod 4$ & \makecell{Balanceable\\ (\Cref{lem-circulant-k-k/2})} & \makecell{Balanceable\\ (\Cref{lem-circulant-4a-odd})} & \multicolumn{2}{c|[1pt]}{Balanceable (\Cref{lem-circulant-4a-even})} \\ \tabucline[1pt]{-}
			
			$\equiv 2 \bmod 4$ & \makecell{Balanceable\\ (\Cref{lem-circulant-k-k/2-odd})} & \makecell{Balanceable\\(\Cref{lem-circulant-4a2-odd})} & \makecell{Balanceable \\ except $C_{6,2}$ \\ (\Cref{lem-circulant-4a2-2})} & \makecell{Balanceable\\ (\Cref{lem-circulant-4a2-even})} \\ \tabucline[1pt]{-}
		\end{tabu}
		\egroup
		\caption{The balanceability of the graph $C_{k,\ell}$. If $\ell > \frac{k}{2}$, then refer to $C_{k,k-\ell}$.}
		\label{tab-balCkl}
	\end{table}
	
	First, we consider the case of $C_{k,\frac{k}{2}}$, which only exist if $k$ is even. For $k \geq 6$, those graphs are exactly the M\"obius ladders. We will distinguish two cases according to the parity of the edge number of  $C_{k,\frac{k}{2}}$, which happens to be even for $k \equiv 0 \bmod 4$, and odd for $k \equiv 2 \bmod 4$. 

	\begin{lemma}
		\label{lem-circulant-k-k/2}
		Let $k$ be a multiple of 4. The graph $C_{k,\frac{k}{2}}$ is balanceable.
	\end{lemma}
	
	\begin{proof}
		Let $k=4a$ with $a>1$. Note that, in $C_{k,\frac{k}{2}}$, every vertex has degree~3. Furthermore, let $e$ be the number of edges in $C_{k,\frac{k}{2}}$, then $e=\frac{3k}{2}=6a$, which implies that $\frac{e}{2}=3a$. Denote the vertices of $C_{k,\frac{k}{2}}$ by $u_0,u_1,u_2,\ldots,u_{k-1}$, and let $I=\{u_0,u_2,\ldots,u_{2a-2}\}$. It is easy to see that $I$ is an independent set of size $a$, and thus $\sum_{v \in I} d(v) = 3a = \frac{e}{2}$. Thus, by \Cref{prop-degreeHalfEdges}, $C_{k,\frac{k}{2}}$ is balanceable.
	\end{proof}
	
	\begin{lemma}
		\label{lem-circulant-k-k/2-odd}
		Let $k$ be an integer such that $k \equiv 2 \bmod 4$. The graph $C_{k,\frac{k}{2}}$ is balanceable.
	\end{lemma}
	
	\begin{proof}
		Let $k=4a+2$ with $a \geq 1$. 

		Let $e$ be the number of edges in $C_{4a+2,2a+1}$, then $e=6a+3$ and thus $\lfloor\frac{e}{2}\rfloor=3a+1$ and $\lceil\frac{e}{2}\rceil=3a+2$. Note also that all vertices have degree~3.
		
		We will begin by proving that we can partition the vertices of $C_{4a+2,2a+1}$ in two sets $X$ and $Y$ such that $e(X,Y)=3a+1$. First, let $X_1 = \{u_0,u_2,\ldots,u_{2a-4}\}$ (if $a=1$ then $X_1$ is empty). Then, let $X_2 = \{u_{2a-2},u_{2a-1}\}$. Now, let $X=X_1 \cup X_2$ and $Y = V(C_{4a+2,2a+1}) \setminus X$. Clearly, the vertices in $X_1$ are an independent set, and are independent from the vertices in $X_2$. Thus, each vertex in $X_1$ has three neighbors in $Y$, and each vertex in $X_2$ has two neighbors in $Y$ (since they are adjacent). This implies that we have $e(X,Y) = 3(a-1)+4 = 3a+1$.
		
		Now, we will prove that there is a subset of vertices $W$ such that $e(G[W]) \in \{3a+1,3a+2\}$. The idea will be to take vertices along the cycle, thus taking a path and some chords within $G[W]$. There are two cases to consider:
		\begin{itemize}
			\item If $a$ is even, then we set $W := \{u_0,u_1,\ldots,u_{2a+\frac{a}{2}+1}\}$. The edges within $G[W]$ are of two sorts: there are $2a+\frac{a}{2}+1$ along the cycle, and there are $\frac{a}{2}+1$ chords (from $u_0u_{2a+1}$ to $u_{\frac{a}{2}}u_{2a+\frac{a}{2}+1}$). Adding those, we have $e(G[W])=2a+\frac{a}{2}+1+\frac{a}{2}+1=3a+2$.
			
			\item If $a$ is odd, then we set $W := \{u_0,u_1,\ldots,u_{2a+\frac{a-1}{2}+1}\}$. The edges within $G[W]$ are of two sorts: there are $2a+\frac{a-1}{2}+1$ along the cycle, and there are $\frac{a-1}{2}+1$ chords (from $u_0u_{2a+1}$ to $u_{\frac{a-1}{2}}u_{2a+\frac{a-1}{2}+1}$). Adding those, we have $e(G[W])=2a+\frac{a-1}{2}+1+\frac{a-1}{2}+1=3a+1$.
		\end{itemize}
		Those two cases allow us to conclude that $C_{4a+2,2a+1}$ verifies the conditions of \Cref{thm-characterization}, and thus $C_{4a+2,2a+1}$ is balanceable.
	\end{proof}

	In all future cases, we assume $\ell<\frac{k}{2}$, thus every vertex of $C_{k,\ell}$ has degree~4. Furthermore, let $e$ be the number of edges in $C_{k,\ell}$, then $e=2k$. We will also denote the vertices of $C_{k,\ell}$ by $u_0,u_1,u_2,\ldots,u_{k-1}$.
	
	\begin{lemma}
		\label{lem-circulant-odd-l}
		Let $k$ be an odd integer. The graph $C_{k,\ell}$ is not balanceable.
	\end{lemma}

	\begin{proof}
		Since $k$ is odd, $\frac{e}{2}=\frac{2k}{2}=k$ is odd. Since $C_{k,\ell}$ is eulerian, \Cref{prop-halfEdgesOddAndEulerian} implies that $C_{k,\ell}$ is not balanceable.
	\end{proof}

	\begin{lemma}
		\label{lem-circulant-4a-odd}
		Let $k$ be an integer such that $k \equiv 0 \bmod 4$, and let $\ell$ be an odd integer. The graph $C_{k,\ell}$ is balanceable.
	\end{lemma}
	
	\begin{proof}
		It is easy to see that $C_{k,\ell}$ is bipartite (the vertices with an even index being one part, and the vertices with an odd index being the other part). Furthermore, it has order $4a$ and is 4-regular. Thus, by \Cref{prop-regularBipartite4n}, $C_{k,\ell}$ is balanceable.
	\end{proof}

	\begin{lemma}
		\label{lem-circulant-4a-even}
		Let $k$ be an integer such that $k \equiv 0 \bmod 4$, and let $\ell$ be an even integer. The graph $C_{k,\ell}$ is balanceable.
	\end{lemma}
	
	\begin{proof}
		We denote $k=4a$ and $\ell=2b$. The idea is to select an independent set $I$ of size $a$, allowing us to invoke \Cref{prop-degreeHalfEdges}. We start by adding to $I$ the vertices $u_0,u_2,\ldots,u_{2b-2}$, thus a set of $b$ independent vertices. However, we cannot add $u_{2b}$ because of the edge $u_0u_{2b}$. Instead, we can add to $I$ the vertices $u_{2b+1},u_{2b+3},\ldots,u_{4b-1}$, thus a set of $b$ new vertices that are independent from each other as well as from the first ones. Again, we have to jump the vertex $u_{4b+1}$ and start from $u_{4b+2}$. By applying this, we will select $\lfloor \frac{a}{b} \rfloor$ such sets of $b$ vertices, and then we can apply the same construction and add the $a-\lfloor\frac{a}{b}\rfloor b$ last vertices we need to have $|I|=a$ (those last vertices will be called \emph{leftover vertices} in the remainder of the proof). This is depicted in \Cref{fig-circulant-4a-even}.
		
		We now need to prove that $I$ is an independent set. Note that the only thing that we need to prove is that the index of the second-biggest-index neighbour of $u_0$ is greater than the index of the last vertex that we selected. Indeed, we selected the sets in such a way that all vertices are independent from each other going forward.
		
		The last vertex that is selected in $I$ with our construction will have index:
		$$ i_{max} = \left\lfloor\frac{a}{b}\right\rfloor2b + \left\lfloor\frac{a}{b}\right\rfloor + 2\left(a - \left\lfloor\frac{a}{b}\right\rfloor b\right) - 1 -1 -r$$
		with $r = 1$ if $\lfloor\frac{a}{b}\rfloor=\frac{a}{b}$ and $r=0$ otherwise.
		
		The $\lfloor\frac{a}{b}\rfloor2b$ are the vertices selected in the $\lfloor\frac{a}{b}\rfloor$ sets that are themselves separated from each other by one supplementary vertex (thus the $\lfloor\frac{a}{b}\rfloor$); then the $2(a - \lfloor\frac{a}{b}\rfloor b)$ are the leftover vertices; and then we have to subtract~1 for the last vertex (which does not count) and again substract~1 for the fact that the indices start at~0. Finally, if $\lfloor\frac{a}{b}\rfloor=\frac{a}{b}$, then there are no leftover vertices and thus we can substract~1 from the total.
		
		Thus, we have $i_{max}=2a+\frac{a}{b}-2-r \leq 2a+\frac{a}{b}-2$. We only have to prove that $i_{max}<4a-2b$ since this would prove that the last vertex that we selected has an index smaller than the second-biggest-index neighbour of $u_0$. There are two cases to consider:
		\begin{enumerate}
			\item If $2b \geq a$, then we have $\frac{a}{b} \leq 2$. Since $\lfloor \frac{a}{b} \rfloor \leq \frac{a}{b}$, we have $i_{max} \leq 2a + \lfloor \frac{a}{b} \rfloor -2 \leq 2a+2-2 = 2a$. Furthermore, since $2b<\frac{k}{2}=2a$ we have $4a-2b>2a$. 
			This implies that $i_{max}<4a-2b$, proving that $I$ is an independent set.
			\item If $2b < a$, then $i_{max} \leq 3a-2$ since $\lfloor\frac{a}{b}\rfloor \leq a$. Furthermore, $4a-2b>4a-a=3a$, and thus $i_{max}<4a-2b$, proving that $I$ is an independent set.
		\end{enumerate}
	
		Thus, $I$ is an independent set of size $a$, and since every vertex has degree~4 we have $\sum_{v \in I} d(v) = 4a = \frac{e}{2}$, and \Cref{prop-degreeHalfEdges} implies that $C_{k,\ell}$ is balanceable.
	\end{proof}
	
	\begin{figure}[h]
		\centering
		\begin{tikzpicture}
			\draw (0,0) circle (4);
			\foreach \I in {0,2,4,6,9,11,13,15,18,20} {
				\pgfmathtruncatemacro{\J}{mod(int(\I+1),40)}
				\draw[ultra thick] (360/40*\I:4) arc (360/40*\I:360/40*\J:4);
				\pgfmathtruncatemacro{\J}{mod(int(\I-1),40)}
				\draw[ultra thick] (360/40*\I:4) arc (360/40*\I:360/40*\J:4);
			}
			\foreach \I in {0,...,39} {
				\node[noeud] (u\I) at (360/40*\I:4) {};
				\node[scale=1] at (360/40*\I:4.4) {$u_{\I}$};
			}
			\foreach \I in {0,...,39} {
				\pgfmathtruncatemacro{\J}{mod(int(\I+8),40)}
				\draw[bend left] (u\I)to(u\J);
			}
			\foreach \I in {0,2,4,6,9,11,13,15,18,20} {
				\node[noeud,fill=black] (u\I) at (360/40*\I:4) {};
				\pgfmathtruncatemacro{\J}{mod(int(\I+8),40)}
				\draw[ultra thick,bend left] (u\I)to(u\J);
			}
			\foreach\I in {32,34,36,38,1,3,5,7,10,12} {
				\pgfmathtruncatemacro{\J}{mod(int(\I+8),40)}
				\draw[ultra thick,bend left] (u\I)to(u\J);
			}
		\end{tikzpicture}
		\caption{A depiction of the proof of \Cref{lem-circulant-4a-even} on $C_{40,8}$. The vertices that we selected in $I$ are in black, and the out-edges of $I$ are bolded.}
		\label{fig-circulant-4a-even}
	\end{figure}
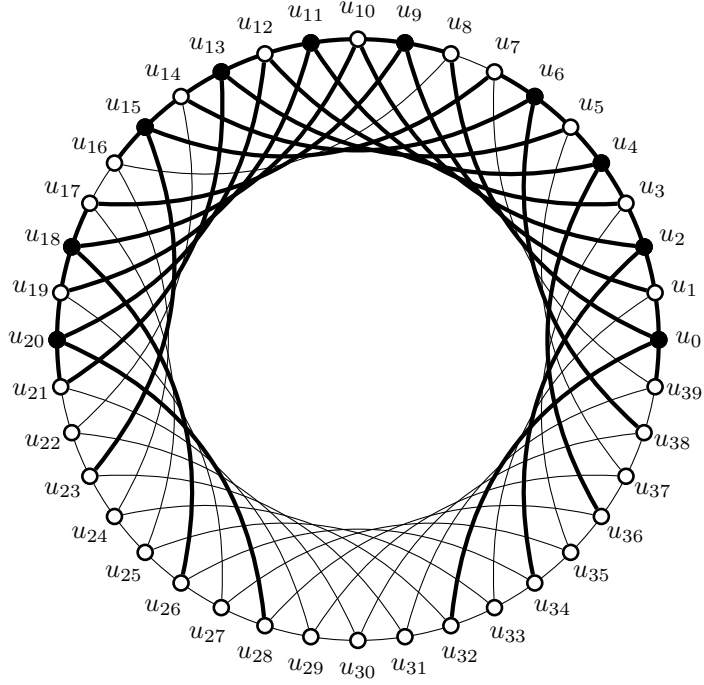

	For the next three lemmas, we cannot construct an independent set $I$ such that $\sum_{v \in I} d(v) = \frac{e}{2}$, since all the degrees are~4 and $\frac{e}{2}$ is not a multiple of~4. We will instead prove that the vertices of $C_{k,\ell}$ can be partitioned in such a way that we can apply \Cref{thm-characterization}.

	\begin{lemma}
		\label{lem-circulant-4a2-odd}
		Let $k$ be an integer such that $k \equiv 2 \bmod 4$, and let $\ell$ be an odd integer with $\ell<\frac{k}{2}$. The graph $C_{k,\ell}$ is balanceable.
	\end{lemma}
	
	\begin{proof}
		Let $k=4a+2$ and $\ell$ be an odd integer. We have two cases to consider.
		
		First, we will prove that we can partition the vertices of $C_{k,\ell}$ in two sets $X$ and $Y$ such that $e(X,Y)=\frac{e}{2}=4a+2$. We begin by setting $X:=\{u_0,u_1\}$, which puts~6 edges between $X$ and $Y$ as long as no neighbour of $u_0$ or $u_1$ is in $Y$. We now select $a-1$ independent vertices that are not neighbours of $u_0$ and $u_1$ and put them in $X$. Note that there are $4a-6$ vertices not neighbours of $u_0$ and $u_1$: $2a-3$ with an even index and $2a-3$ with an odd index. Thus, we can select $a-1$ vertices of even index (without loss of generality), which is always possible. Indeed, assume by contradiction that $a-1>2a-3$; then $a<2$, \emph{i.e.} $k<10$, \emph{i.e.} $k=6$, which is a contradiction since $\ell$ is odd, but $\ell>1$ and $\ell<\frac{k}{2}=3$ so this case cannot occur. Since the $a-1$ vertices we just added to $X$ are independent, we have $e(X,Y)=6+4(a-1)=4a+2$. This construction is depicted in \Cref{fig-circulant-4a2-odd-1}.
		
		Now, denote $C_{k,\ell}$ by $G$, its vertex-set by $V$ and its edge-set by $E$. We will prove that we can partition $V$ in two sets $W$ and $V \setminus W$ such that $e(G[W])=\frac{e}{2}=4a+2$. Note that two adjacent vertices in $V \setminus W$ independent from all other vertices in $V \setminus W$ put~7 edges in $E \setminus E(G[W])$. We will construct $V \setminus W$ by selecting two pairs of adjacent vertices independent from each other, and $a-3$ independent vertices that will not be neighbours of the four vertices previously selected. Thus, we will have $e-e(G[W])=2 \times 7+4(a-3)=4a+2$, and thus $e(G[W])=4a+2$. There are three cases to consider.
		
		First, assume that $k=10$, the only graph to consider is $C_{10,3}$. In this case, we cannot construct the sets as explained above (since $2\times7=14>10=\frac{e}{2}$). However, by setting $W=\{u_0,u_1,u_2,u_3,u_4,u_5,u_6\}$, there are~10 edges in $G[W]$, so this case is covered.
		
		Then, assume that $\ell=3$ and $k>10$. We put $u_0,u_1,u_5$ and $u_6$ in $V \setminus W$. There are $k-13=4a-11$ vertices that are neither those nor neighbours of those: $2a-5$ with an even index and $2a-6$ with an odd index. Thus, we can select $a-3$ independent vertices with an even index, which is always possible since $a-3>2a-5$ if and only if $a<2$, \emph{i.e.} $k<10$, which cannot occur as discussed previously. This implies that we have $e-e(G[W])=4a+2$. This construction is depicted on the left-hand side of \Cref{fig-circulant-4a2-odd-2}.
		
		Finally, assume that $\ell>3$ and $k>10$. We put $u_0,u_1,u_3$ and $u_4$ in $V \setminus W$. There are $k-15=4a-13$ vertices that are neither those nor neighbours of those: $2a-6$ with an even index and $2a-7$ with an odd index. Thus, we can select $a-3$ independent vertices with an even index, which is always possible. Indeed, assume by contradiction that $a-3>2a-6$, then $a<3$, \emph{i.e.} $k<14$; the case $k=6$ has been discussed previously, and the case $k=10$ cannot occur either since this would imply $\ell=5=\frac{k}{2}$, a contradiction. This implies that we have $e-e(G[W])=4a+2$. This construction is depicted on the right-hand side of \Cref{fig-circulant-4a2-odd-2}.
		
		Altogether, this allows us to invoke \Cref{thm-characterization}, and thus to conclude that the circulant graph $C_{4a+2,\ell}$ is balanceable when $\ell$ is odd and $\ell < 2a+1$.
	\end{proof}
	
	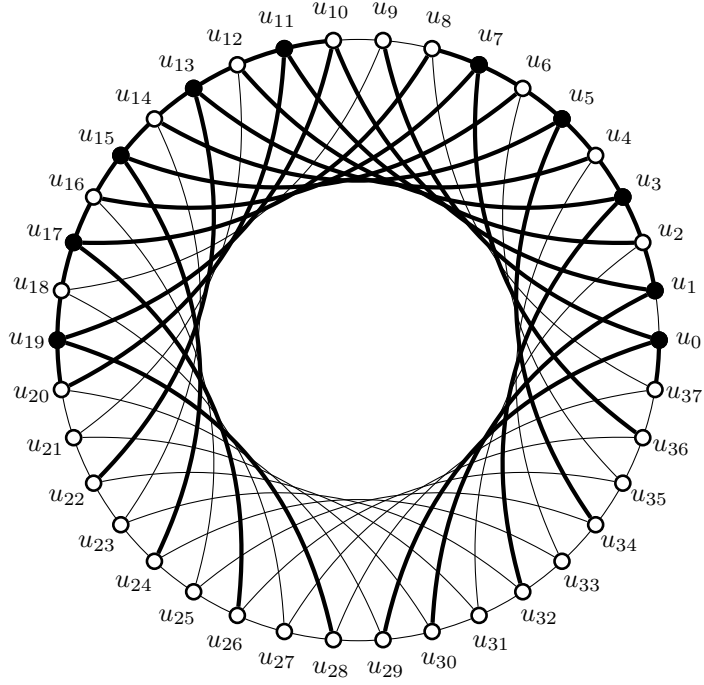
\begin{figure}[h]
		\centering
		\begin{tikzpicture}
			\draw (0,0) circle (4);
			\foreach \I in {3,5,7,11,13,15,17,19} {
				\pgfmathtruncatemacro{\J}{mod(int(\I+1),38)}
				\draw[ultra thick] (360/38*\I:4) arc (360/38*\I:360/38*\J:4);
				\pgfmathtruncatemacro{\J}{mod(int(\I-1),38)}
				\draw[ultra thick] (360/38*\I:4) arc (360/38*\I:360/38*\J:4);
			}
			\draw[ultra thick] (360/38*1:4) arc (360/38*1:360/38*2:4);
			\draw[ultra thick] (360/38*37:4) arc (360/38*37:360/38*38:4);
			\foreach \I in {0,...,37} {
				\node[noeud] (u\I) at (360/38*\I:4) {};
				\node[scale=1] at (360/38*\I:4.4) {$u_{\I}$};
			}
			\foreach \I in {0,...,37} {
				\pgfmathtruncatemacro{\J}{mod(int(\I+9),38)}
				\draw[bend left] (u\I)to(u\J);
			}
			\foreach \I in {0,1,3,5,7,11,13,15,17,19} {
				\node[noeud,fill=black] (u\I) at (360/38*\I:4) {};
				\pgfmathtruncatemacro{\J}{mod(int(\I+9),38)}
				\draw[ultra thick,bend left] (u\I)to(u\J);
			}
			\foreach\I in {29,30,32,34,36,2,4,6,8,10} {
				\pgfmathtruncatemacro{\J}{mod(int(\I+9),38)}
				\draw[ultra thick,bend left] (u\I)to(u\J);
			}
		\end{tikzpicture}
		\caption{A depiction of the first case of the proof of \Cref{lem-circulant-4a2-odd} on $C_{38,9}$. The vertices that we selected in $X$ are in black, and the edges between $X$ and $Y$ are bolded.}
		\label{fig-circulant-4a2-odd-1}
	\end{figure}
	
	\begin{figure}[h]
		\begin{bigcenter}
			\begin{tikzpicture}
			
			\node (lIs3) at (0,0) {
				\begin{tikzpicture}
				\draw (0,0) circle (4);
				\foreach \I in {0,1,5,6,10,12,14,16,18,20} {
					\pgfmathtruncatemacro{\J}{mod(int(\I+1),38)}
					\draw[ultra thick] (360/38*\I:4) arc (360/38*\I:360/38*\J:4);
					\pgfmathtruncatemacro{\J}{mod(int(\I-1),38)}
					\draw[ultra thick] (360/38*\I:4) arc (360/38*\I:360/38*\J:4);
				}
				\foreach \I in {0,...,37} {
					\node[noeud] (u\I) at (360/38*\I:4) {};
					\node[scale=1] at (360/38*\I:4.4) {$u_{\I}$};
				}
				\foreach \I in {0,...,37} {
					\pgfmathtruncatemacro{\J}{mod(int(\I+3),38)}
					\draw[bend left] (u\I)to(u\J);
				}
				\foreach \I in {0,1,5,6,10,12,14,16,18,20} {
					\node[noeud,fill=black] (u\I) at (360/38*\I:4) {};
					\pgfmathtruncatemacro{\J}{mod(int(\I+3),38)}
					\draw[ultra thick,bend left] (u\I)to(u\J);
				}
				\foreach\I in {35,36,2,3,7,9,11,13,15,17} {
					\pgfmathtruncatemacro{\J}{mod(int(\I+3),38)}
					\draw[ultra thick,bend left] (u\I)to(u\J);
				}
				\end{tikzpicture}
			};
		
			\node (lIsMore) at (10,0) {
				\begin{tikzpicture}
				\draw (0,0) circle (4);
				\foreach \I in {0,1,3,4,6,8,14,16,18,20} {
					\pgfmathtruncatemacro{\J}{mod(int(\I+1),38)}
					\draw[ultra thick] (360/38*\I:4) arc (360/38*\I:360/38*\J:4);
					\pgfmathtruncatemacro{\J}{mod(int(\I-1),38)}
					\draw[ultra thick] (360/38*\I:4) arc (360/38*\I:360/38*\J:4);
				}
				\foreach \I in {0,...,37} {
					\node[noeud] (u\I) at (360/38*\I:4) {};
					\node[scale=1] at (360/38*\I:4.4) {$u_{\I}$};
				}
				\foreach \I in {0,...,37} {
					\pgfmathtruncatemacro{\J}{mod(int(\I+9),38)}
					\draw[bend left] (u\I)to(u\J);
				}
				\foreach \I in {0,1,3,4,6,8,14,16,18,20} {
					\node[noeud,fill=black] (u\I) at (360/38*\I:4) {};
					\pgfmathtruncatemacro{\J}{mod(int(\I+9),38)}
					\draw[ultra thick,bend left] (u\I)to(u\J);
				}
				\foreach\I in {29,30,32,33,35,37,5,7,9,11} {
					\pgfmathtruncatemacro{\J}{mod(int(\I+9),38)}
					\draw[ultra thick,bend left] (u\I)to(u\J);
				}
				\end{tikzpicture}
			};
			
			\end{tikzpicture}
			\caption{A depiction of the second case of the proof of \Cref{lem-circulant-4a2-odd} on $C_{38,3}$ and $C_{38,9}$. The vertices that we selected in $V \setminus W$ are in black, and the edges outside of $G[W]$ are bolded.}
			\label{fig-circulant-4a2-odd-2}
		\end{bigcenter}
	\end{figure}

	\begin{lemma}
		\label{lem-circulant-4a2-2}
		Let $k$ be an integer such that $k \equiv 2 \bmod 4$. The graph $C_{k,2}$ is balanceable if and only if $k \neq 6$.
	\end{lemma}
	
	\begin{proof}
		This proof contains two parts: first, we will prove that $C_{6,2}$ is not balanceable; then, we will prove that $C_{4a+2,2}$ is balanceable when $a>1$.
		
		First, assume that $k=6$. We will prove that there is no subset of vertices $W$ such that $e(G[W])=6$. First, note that $W$ cannot possibly be empty or all the vertices. Taking this into account, \Cref{tab-circulant-6-2} shows possible sets for different sizes of $W$ as well as $e(G[W])$ in each case (the possible sets are up to renaming vertices). Since no set $W$ gives $e(G[W])=6$, \Cref{thm-characterization} implies that $C_{6,2}$ is not balanceable.
		
		\begin{table}[h]
			\begin{bigcenter}
				\begin{tabu} {|[1pt]c|[1pt]c|c|c|c|c|}
					\tabucline[1pt]{-}
					$|W|$ & 1 & 2 & 3 & 4 & 5 \\ \tabucline[1pt]{-}
					\makecell{Possible\\vertices in $W$\\$\rightarrow e(G[W])$} & $u_0 \rightarrow 0$ & \makecell{$u_0,u_1 \rightarrow 1$\\$u_0,u_2 \rightarrow 1$\\$u_0,u_3 \rightarrow 0$} & \makecell{$u_0,u_1,u_2 \rightarrow 3$\\$u_0,u_1,u_3 \rightarrow 2$\\$u_0,u_2,u_4 \rightarrow 3$} & \makecell{$u_0,u_1,u_2,u_3 \rightarrow 5$\\$u_0,u_1,u_2,u_4 \rightarrow 5$\\$u_0,u_1,u_3,u_4 \rightarrow 4$} & $u_0,\ldots,u_4 \rightarrow 8$ \\ \tabucline[1pt]{-}
				\end{tabu}
				\caption{Possible sets $W$ of vertices of $C_{6,2}$ (up to renaming the vertices), and the value of $e(G[W])$ for each of them.}
				\label{tab-circulant-6-2}
			\end{bigcenter}
		\end{table}
		
		Now, assume that $a>1$, we will prove that $C_{4a+2,2}$ is balanceable. The proof is similar to the proof of \Cref{lem-circulant-4a2-odd}.
		
		First, we will prove that we can partition the vertices of $C_{4a+2,2}$ in two sets $X$ and $Y$ such that $e(X,Y)=4a+2$. For this, we set $X_1=\{u_0,u_3,\ldots,u_{3(a-2)}\}$, $X_2=\{u_{4a-1},u_{4a}\}$, and $X = X_1 \cup X_2$. It is easy to see that $X_1$ is an independent set and that no vertex in $X_2$ is adjacent to a vertex in $X_1$ (since we have $a>1 \Rightarrow 4a-3>3a-2>3(a-2)$). Thus, we have $e(X,Y)=e(X_1,Y)+e(X_2,Y)=4(a-1)+6=4a+2$.
		
		Then, as in the proof of \Cref{lem-circulant-4a2-odd}, assume that $a=2$, \emph{i.e.} $k=10$. If we set $W=\{u_0,u_1,u_2,u_3,u_4,u_6,u_8\}$, then we have $e(G[W])=10$. This, with the previous point (that applies if $a=2$), proves that $C_{10,2}$ is balanceable. Assume now that $a>2$. Let $V_1 = \{u_0,u_3,\ldots,u_{3(a-4)}\}$ (if $a=3$ then we set $V_1 = \emptyset$), $V_2 = \{u_{4a-6},u_{4a-5}\}$ and $V_3=\{u_{4a-2},u_{4a-1}\}$; then set $W = V \setminus (V_1 \cup V_2 \cup V_3)$. It is easy to see that $V_1$ is an independent set and that no vertex in $V_2$ (resp. $V_3$) is adjacent to a vertex in $V_1$ or $V_3$ (resp. $V_1$ or $V_2$), since we have $a>2 \Rightarrow 4a-8>3a-6>3(a-4)$. Thus, we have $e-e(W)=4(a-3)+14=4a+2$, which implies $e(G[W])=4a+2$.
		
		The above constructions allow us to invoke \Cref{thm-characterization}, which implies that $C_{4a+2,2}$ is balanceable.
	\end{proof}

	\begin{lemma}
		\label{lem-circulant-4a2-even}
		Let $k$ be an integer such that $k \equiv 2 \bmod 4$, and let $\ell$ be an even integer such that $\ell>2$. The graph $C_{k,\ell}$ is balanceable.
	\end{lemma}
	
	\begin{proof}
		Let $k=4a+2$ and $\ell=2b$.
		The proof for this lemma is a mix of the proofs for \Cref{lem-circulant-4a-even,lem-circulant-4a2-2}: we will use the structure we constructed in the proof for \Cref{lem-circulant-4a-even} and add either one or two independent edges to it, modifying the structure to keep everything independent from each other.
		
		First, we construct $X$ in two steps. We begin by creating a set $X_1$ by applying the same construction than in the proof of \Cref{lem-circulant-4a-even} (so several sets of $\frac{\ell}{2}$ vertices at distance~2 along the outer cycle from each other, each set being separated from the others by another vertex): a total of $a-1$ such vertices are added to $X_1$. Then, let $X_2 := \{u_{4a-1},u_{4a}\}$. Now, we need $X_1$ and $X_2$ to be independent from each other, so if a vertex in $X_1$ is adjacent to a vertex in $X_2$ (this can happen to at most one vertex), we remove it from $X_1$ and add to this set the next vertex in the construction described in the proof of \Cref{lem-circulant-4a-even} (we may start a new set this way). This is depicted on the left-hand side of \Cref{fig-circulant-4a2-even}.
		The last vertex that is selected in $X_1$ with our construction will have index:
		$$ i_{max} \leq 2a + \left\lfloor\frac{a}{b}\right\rfloor-2+3-2.$$
		That is, the same maximum index than in the proof of \Cref{lem-circulant-4a-even}, but with two corrections: $+3$ may happen since we could start a new set by shifting the neighbour of either $u_{4a-1}$ or $u_{4a}$ (this gives us~$+2$, and may give us an additional~$+1$ if the vertex we shift creates a new set), and $-2$ since we only need $a-1$ vertices in $X_1$ (instead of the $a$ from the proof of \Cref{lem-circulant-4a-even}). Now, we need to prove that this last index is less than $4a-1-\ell$.
		
		If $\ell \geq a$, then we can check that we will always have $i_{max}=2a-3 < 2a-2 < 4a-1-\ell$ since $\ell < \frac{k}{2} = 2a+1$. Indeed, we will put in $X_1$ first the $b-1$ vertices $u_0,u_2,\ldots,u_{\ell-4}$, then the $a-b$ vertices $u_{\ell-1},u_{\ell+1},\ldots,u_{\ell+2(a-b)-3}$ (which is always possible since $\ell \geq a$). The last index will thus always be $\ell+2(a-b)-3=2a-3$.
		
		Assume now that $\ell<a$. Since $\ell>2$, we have $b\geq 2$ and thus $\lfloor\frac{a}{b}\rfloor \leq \lfloor\frac{a}{2}\rfloor \leq \frac{a}{2}+1$. Thus, $i_{max} \leq 2a+\lfloor\frac{a}{b}\rfloor-1 \leq 2a + \frac{a}{2} +1-1 = \frac{5a}{2} < 3a \leq 4a-1-\ell$.
		
		Hence, in this construction, $X_1$ and $X_2$ are independent from each other, and by setting $X = X_1 \cup X_2$ we have $e(X,Y)=6+4(a-1)=4a+2$.
		
		Now, as in the proof of \Cref{lem-circulant-4a2-odd}, we have to deal with the case of $C_{10,4}$. In this case, by setting $W:=\{u_0,u_1,u_2,u_4,u_5,u_6,u_8\}$, we have $e(G[W])=10$.
		
		Finally, for $k \geq 14$, we construct $V \setminus W$ by applying the same construction than in the proof of \Cref{lem-circulant-4a-even}. Let $V_1$ be a set of $a-3$ vertices constructed this way, then let either $V_2=\{u_{4a-4,4a-3}\}$ (if $\ell > 4$) or $V_2 = \{u_{4a-7},u_{4a-6}\}$ (if $\ell=4$), and $V_3=\{u_{4a-1},u_{4a}\}$. Again, we shift the potential vertices in $V_1$ adjacent to a vertex in $V_2$ or $V_3$ (at most~2 such vertices), and thus the highest index we can reach is:
		$$ i_{max} \leq 2a + \left\lfloor\frac{a}{b}\right\rfloor-2+6-6.$$
		The $+6$ comes from the two potential shifts, and the $-6$ from the fact that we select $a-3$ vertices instead of $a$. We now need to verify that $i_{max}<4a-i-\ell$ for $i \in \{4,7\}$ (depending on the value of $\ell$). We have three cases to check:
		\begin{enumerate}
			\item If $\ell=4$, then $i_{max} \leq 2a+\lfloor\frac{a}{2}\rfloor-2$; and $4a-7-\ell=4a-11$. Now, if $a > 6$ then since $\lfloor \frac{a}{2} \rfloor \leq \frac{a}{2}$ it is easy to check that $i_{max} \leq 2a + \frac{a}{2} -2 < 4a-11$. We need to check that $i_{max}<4a-11$ in the remaining cases:
			\begin{enumerate}
				\item If $a=3$, then we have $V_1=\emptyset$ so no contradiction arises;
				\item If $a=4$, then we have $i_{max}=0$ and $4a-11=5$ so no contradiction arises;
				\item If $a=5$, then we have $i_{max}=3$ and $4a-11=9$ so no contradiction arises;
				\item If $a=6$, then we have $i_{max}=5$ and $4a-11=13$ so no contradiction arises.
			\end{enumerate}
			Thus, if $\ell=4$, then $i_{max}<4a-7-\ell$.
			\item If $\ell \geq a$, then we can check that we will always have $i_{max}=2a-5$ and $4a-4-\ell \geq 2a-4$ since $\ell < \frac{k}{2} = 2a+1$. Indeed, we will put in $V_1$ first the $b-3$ vertices $u_0,u_2,\ldots,u_{\ell-8}$ as well as $u_{\ell-4}$, then the $a-b-1$ vertices $u_{\ell-1},u_{\ell+1},\ldots,u_{\ell+2(a-b)-5}$ (which is always possible since $\ell \geq a$). The last index will thus always be $\ell+2(a-b)-5=2a-5$.
			\item If $\ell>4$ (thus $b>2$) and $\ell<a$, then since $\lfloor \frac{a}{b} \rfloor < \frac{a}{2}$ we have $i_{max} \leq 2a+\lfloor \frac{a}{b} \rfloor-2 < \frac{5a}{2}-2$; and $4a-4-\ell > 3a-4$. Now, we know that $a>\ell>4$, so it is easy to check that $\frac{5a}{2}-2<3a-4$, and thus, that $i_{max}<4a-4-\ell$.
		\end{enumerate}
		All those cases prove that $V_2$ and $V_3$ are independent from $V_1$. By setting $V \setminus W = V_1 \cup V_2 \cup V_3$, we have $e-e(G[W])=14+4(a-3)=4a+2$, and thus $e(G[W])=4a+2$.
		
		Those two constructions, depicted in \Cref{fig-circulant-4a2-even}, allow us to invoke \Cref{thm-characterization}, which implies that $C_{4a+2,\ell}$ is balanceable when $\ell$ is even and $\ell>2$.
	\end{proof}

	\begin{figure}[h]
		\begin{bigcenter}
			\begin{tikzpicture}
				\node (case1) at (0,0) {
					\begin{tikzpicture}
					\draw (0,0) circle (4);
					\foreach \I in {0,2,4,7,9,11,13,16} {
						\pgfmathtruncatemacro{\J}{mod(int(\I+1),38)}
						\draw[ultra thick] (360/38*\I:4) arc (360/38*\I:360/38*\J:4);
						\pgfmathtruncatemacro{\J}{mod(int(\I-1),38)}
						\draw[ultra thick] (360/38*\I:4) arc (360/38*\I:360/38*\J:4);
					}
					\draw[ultra thick] (360/38*36:4) arc (360/38*36:360/38*37:4);
					\draw[ultra thick] (360/38*34:4) arc (360/38*34:360/38*35:4);
					\foreach \I in {0,...,37} {
						\node[noeud] (u\I) at (360/38*\I:4) {};
						\node[scale=1] at (360/38*\I:4.4) {$u_{\I}$};
					}
					\foreach \I in {0,...,37} {
						\pgfmathtruncatemacro{\J}{mod(int(\I+8),38)}
						\draw[bend left] (u\I)to(u\J);
					}
					\foreach \I in {0,2,4,7,9,11,13,16,35,36} {
						\node[noeud,fill=black] (u\I) at (360/38*\I:4) {};
						\pgfmathtruncatemacro{\J}{mod(int(\I+8),38)}
						\draw[ultra thick,bend left] (u\I)to(u\J);
					}
					\foreach\I in {30,32,34,37,1,3,5,8,27,28} {
						\pgfmathtruncatemacro{\J}{mod(int(\I+8),38)}
						\draw[ultra thick,bend left] (u\I)to(u\J);
					}
					\end{tikzpicture}
				};
				
				\node (case2) at (10,0) {
					\begin{tikzpicture}
					\draw (0,0) circle (4);
					\foreach \I in {0,4,7,9,11,13,32,33,35,36} {
						\pgfmathtruncatemacro{\J}{mod(int(\I+1),38)}
						\draw[ultra thick] (360/38*\I:4) arc (360/38*\I:360/38*\J:4);
						\pgfmathtruncatemacro{\J}{mod(int(\I-1),38)}
						\draw[ultra thick] (360/38*\I:4) arc (360/38*\I:360/38*\J:4);
					}
					\foreach \I in {0,...,37} {
						\node[noeud] (u\I) at (360/38*\I:4) {};
						\node[scale=1] at (360/38*\I:4.4) {$u_{\I}$};
					}
					\foreach \I in {0,...,37} {
						\pgfmathtruncatemacro{\J}{mod(int(\I+8),38)}
						\draw[bend left] (u\I)to(u\J);
					}
					\foreach \I in {0,4,7,9,11,13,32,33,35,36} {
						\node[noeud,fill=black] (u\I) at (360/38*\I:4) {};
						\pgfmathtruncatemacro{\J}{mod(int(\I+8),38)}
						\draw[ultra thick,bend left] (u\I)to(u\J);
					}
					\foreach\I in {30,34,37,1,3,5,24,25,27,28} {
						\pgfmathtruncatemacro{\J}{mod(int(\I+8),38)}
						\draw[ultra thick,bend left] (u\I)to(u\J);
					}
					\end{tikzpicture}
				};
			\end{tikzpicture}
			\caption{A depiction of the proof of \Cref{lem-circulant-4a2-even} on $C_{38,8}$. On the left-hand side, the vertices in $X$ are bolded, as well as the edges between $X$ and $Y$. On the right-hand side, the vertices in $V \setminus W$ are bolded, as well as the edges outside $G[W]$.}
			\label{fig-circulant-4a2-even}
		\end{bigcenter}
	\end{figure}

	Together, \Cref{lem-circulant-k-k/2,lem-circulant-k-k/2-odd,lem-circulant-odd-l,lem-circulant-4a-odd,lem-circulant-4a-even,lem-circulant-4a2-odd,lem-circulant-4a2-2,lem-circulant-4a2-even} prove the validity of \Cref{thm-circulant}, which fully characterizes, among this special class of circulant graphs, which are balanceable and which are not. In particular, note that every M\"obius ladder is balanceable (by \Cref{lem-circulant-k-k/2,lem-circulant-k-k/2-odd}) and that the 3-antiprism graph is the only non-balanceable antiprism graph (by \Cref{lem-circulant-4a2-2} and a special case of \Cref{lem-circulant-4a-even}).
	
	With this special class of circulant graphs, we get one step closer towards a full classification of the balanceability of circulant graphs, which had been informally started by the study of complete graphs and cycles. This full classification is an interesting open problem for future studies.
	
	\section{Balanceability of grids}
	\label{sec-grids}
	
	In this section, we study the balanceability of grid graphs with an even edge number. In particular, we study rectangular and triangular grids.
	
	\subsection{Rectangular grids}
	\label{subsec-rectangularGrids}
	
	Let $G_{k,\ell}$ be the rectangular grid graph with $k$ vertices per row and $\ell$ vertices per column. It is easy to see that $G_{k,\ell}$ has $k(\ell-1)+(k-1)\ell=2k\ell-(k+\ell)$ edges, and this number is even if and only if $k$ and $\ell$ have the same parity.
	
	\begin{theorem}
		\label{thm-rectangularGrids}
		Let $k$  and $\ell$ be two integers such that $k,\ell > 1$.
		If $k$ and $\ell$ have the same parity, then $G_{k,\ell}$ is balanceable.
	\end{theorem}
	
	\begin{proof}
		In the grid graph $G_{k,\ell}$ with vertex-set $V$ and edge-set $E$, vertices can have degree two, three or four. The repartition is as follows:
		\begin{itemize}
			\item $4$ vertices of degree two (the corners);
			\item $2(k-2)+2(\ell-2)=2(k+\ell)-8$ vertices of degree three (the sides);
			\item $k\ell - 2(k+\ell)+4$ vertices of degree four (the inside).
		\end{itemize}
		It is well-known that $\sum_{v \in V} d(v)=2|E|$. We want to find an independent set of vertices $I$ such that $\sum_{v \in I} d(v) = \frac{|E|}{2}$. To do this, we can select one fourth of the vertices in every degree set. There are several cases.
		\ \\
		
		\noindent\textbf{Case 1:} If $k$ and $\ell$ are even, then we can select 1 vertex of degree two, $\frac{k+\ell}{2}-2$ vertices of degree three, and $\frac{k\ell}{4}-\frac{k+\ell}{2}+1$ vertex of degree four. An example is depicted on \Cref{fig-g48}. It is always possible to select those vertices such that they induce an independent set, since there is an independent set containing half the vertices of $G_{k,\ell}$, and in particular half the corners, half of the sides and half of the inside. By applying \Cref{prop-degreeHalfEdges}, $G_{k,\ell}$ is balanceable.
		\ \\
		
		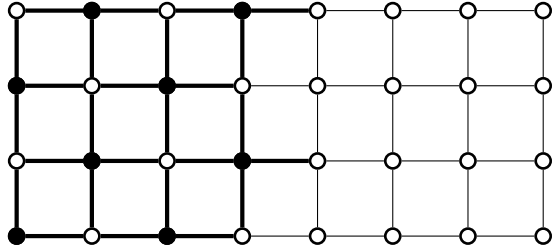
\begin{figure}[h]
			\centering
			\begin{tikzpicture}
				\foreach \I in {0,1,2,3} {
					\foreach \J in {0,1,...,7} {
						\node[noeud] (u\I\J) at (\J,\I) {};
					}
				}
				\foreach \I in {0,1,2,3} {
					\foreach[evaluate=\J as \K using {int(\J+1)}] \J in {0,...,6} {
						\draw (u\I\J) to (u\I\K);
					}
				}
				\foreach \J in {0,...,7} {
					\foreach[evaluate=\I as \K using {int(\I+1)}] \I in {0,1,2} {
						\draw (u\I\J) to (u\K\J);
					}
				}
				
				\foreach \I/\J in {0/0,2/0,1/1,3/1,2/2,1/3,3/3,0/2} {
					\node[noeud,fill=black] (u\I\J) at (\J,\I) {};
				}
				\foreach \I/\J in {00/01,00/10,20/10,20/21,20/30,11/01,11/10,11/21,11/12,31/30,31/21,31/32,22/21,22/12,22/32,22/23,13/12,13/03,13/23,13/14,33/32,33/23,33/34,02/01,02/12,02/03} {
					\draw[ultra thick] (u\I)to(u\J);
				}
			\end{tikzpicture}
			\caption{The independent set $I$ such that $\sum_{v \in I} d(v) = \frac{|E(G)|}{2}$ for $G_{4,8}$. Vertices in $I$ are bolded, as well as the out-edges of $I$.}
			\label{fig-g48}
		\end{figure}
		
		\noindent\textbf{Case 2:} If $k$ and $\ell$ are odd, and $k+\ell$ is not a multiple of~4, then we can select 1 vertex of degree two, $\frac{k+\ell}{2}-3$ vertices of degree three, and $\frac{k\ell+7}{4}-\frac{k+\ell}{2}$ vertices of degree four. An example is depicted on \Cref{fig-g37}. Again, it is always possible to select those vertices such that they induce an independent set (by the same argument than the previous case).
		Furthermore, $k\ell+7$ is a multiple of~4: by noting $k=2a+1$ and $\ell=2b+1$, we have $k\ell+7=4ab+2a+2b+1+7=4ab+8+(2a+2b)=4ab+8+(k+\ell-2)$, and the fact that $k+\ell$ is not a multiple of~4 implies that $(k+\ell-2)$ is.
		We will thus have:
		\begin{eqnarray*}
			\sum_{v \in I} d(v) & = & 2 + 3\left(\frac{k+\ell}{2}-3\right) + 4 \left( \frac{k\ell+7}{4} - \frac{k+\ell}{2} \right) \\
			& = & 2 + 3\frac{k+\ell}{2} - 9 + k\ell + 7 - 4\frac{k+\ell}{2} \\
			& = & k\ell - \frac{k+\ell}{2} \\
			& = & \frac{|E|}{2}
		\end{eqnarray*}
	
		\Cref{prop-degreeHalfEdges} then implies that $G_{k,\ell}$ is balanceable.
		
		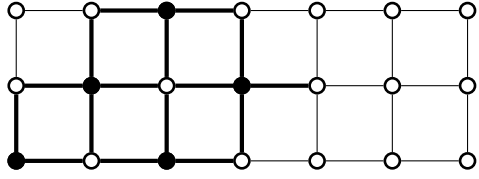
\begin{figure}[h]
			\centering
			\begin{tikzpicture}
				\foreach \I in {0,1,2} {
					\foreach \J in {0,1,...,6} {
						\node[noeud] (u\I\J) at (\J,\I) {};
					}
				}
				\foreach \I in {0,1,2} {
					\foreach[evaluate=\J as \K using {int(\J+1)}] \J in {0,...,5} {
						\draw (u\I\J) to (u\I\K);
					}
				}
				\foreach \J in {0,...,6} {
					\foreach[evaluate=\I as \K using {int(\I+1)}] \I in {0,1} {
						\draw (u\I\J) to (u\K\J);
					}
				}
				
				\foreach \I/\J in {0/0,1/1,2/2,1/3,0/2} {
					\node[noeud,fill=black] (u\I\J) at (\J,\I) {};
				}
				\foreach \I/\J in {00/01,00/10,11/01,11/10,11/21,11/12,22/21,22/12,22/23,13/12,13/03,13/23,13/14,02/01,02/12,02/03} {
					\draw[ultra thick] (u\I)to(u\J);
				}
			\end{tikzpicture}
			\caption{The independent set $I$ such that $\sum_{v \in I} d(v) = \frac{|E(G)|}{2}$ for $G_{3,7}$. Vertices in $I$ are bolded, as well as the out-edges of $I$.}
			\label{fig-g37}
		\end{figure}
		
		\noindent\textbf{Case 3:} If $k$ and $\ell$ are odd, and $k+\ell$ is a multiple of~4, then we can select 2 vertices of degree two, $\frac{k+\ell}{2}-3$ vertices of degree three, and $\frac{k\ell+5}{4}-\frac{k+\ell}{2}$ vertices of degree four. An example is depicted on \Cref{fig-g57}. Again, it is always possible to select those vertices such that they induce an independent set (by the same argument than the previous case).
		Furthermore, $k\ell+5$ is a multiple of~4: by noting $k=2a+1$ and $\ell=2b+1$, we have $k\ell+5=4ab+2a+2b+1+5=4ab+(2a+2b+6)=4ab+(k+\ell+4)$.
		We will thus have:
		\begin{eqnarray*}
			\sum_{v \in I} d(v) & = & 4 + 3\left(\frac{k+\ell}{2}-3\right) + 4 \left( \frac{k\ell+5}{4} - \frac{k+\ell}{2} \right) \\
			& = & 4 + 3\frac{k+\ell}{2} - 9 + k\ell + 5 - 4\frac{k+\ell}{2} \\
			& = & k\ell - \frac{k+\ell}{2} \\
			& = & \frac{|E|}{2}
		\end{eqnarray*}
		\Cref{prop-degreeHalfEdges} then implies that $G_{k,\ell}$ is balanceable.
		
		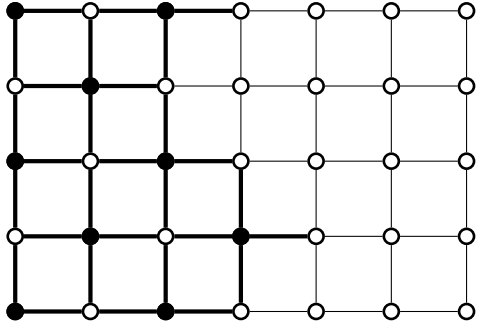
\begin{figure}[h]
			\centering
			\begin{tikzpicture}
				\foreach \I in {0,1,2,3,4} {
					\foreach \J in {0,1,...,6} {
						\node[noeud] (u\I\J) at (\J,\I) {};
					}
				}
				\foreach \I in {0,1,2,3,4} {
					\foreach[evaluate=\J as \K using {int(\J+1)}] \J in {0,...,5} {
						\draw (u\I\J) to (u\I\K);
					}
				}
				\foreach \J in {0,...,6} {
					\foreach[evaluate=\I as \K using {int(\I+1)}] \I in {0,1,2,3} {
						\draw (u\I\J) to (u\K\J);
					}
				}
				
				\foreach \I/\J in {0/0,1/1,2/2,1/3,0/2,2/0,3/1,4/0,4/2} {
					\node[noeud,fill=black] (u\I\J) at (\J,\I) {};
				}
				\foreach \I/\J in {00/01,00/10,11/01,11/10,11/21,11/12,22/21,22/12,22/32,22/23,13/12,13/03,13/23,13/14,02/01,02/12,02/03,20/10,20/21,20/30,31/30,31/21,31/41,31/32,40/30,40/41,42/41,42/32,42/43} {
					\draw[ultra thick] (u\I)to(u\J);
				}
			\end{tikzpicture}
			\caption{The independent set $I$ such that $\sum_{v \in I} d(v) = \frac{|E(G)|}{2}$ for $G_{5,7}$. Vertices in $I$ are bolded, as well as the out-edges of $I$.}
			\label{fig-g57}
		\end{figure}
		
		All possible cases have been covered, and thus if $k$ and $\ell$ have the same parity, then the rectangular grid $G_{k,\ell}$ is balanceable.
	\end{proof}
	
	\subsection{Triangular grids}
	\label{subsec-triangularGrids}
	
	Let $T_h$ be the (equilateral) triangular grid with $h$ vertices on each side. It is easy to see that $T_h$ has $\frac{3(h-1)h}{2}$ edges, and that this is even if and only if $h \bmod 8 \in \{0,1,4,5\}$. We will prove that some triangular grids are not balanceable, while others are.
	
	\begin{theorem}
		\label{thm-triangularGrids}
		Let $h$ be a positive integer such that $h \bmod 8 \in \{0,1,4,5\}$.
		The triangular grid $T_h$ is balanceable if and only if $h \bmod 8 \in \{0,1\}$.
	\end{theorem}
	
	\begin{proof}
		We prove two statements here: the non-balanceability of $T_{8k+4}$ and $T_{8k+5}$; as well as the balanceability of $T_{8k}$ and $T_{8k+1}$.
		
		We will consider the vertices of $T_h$ row by row, starting from a single vertex at the top of the grid. The vertex $u_i^j$ will be the $i$th vertex (starting from the left) in the $j$th row (starting from the top), so the top vertex is $u_1^1$, the second row contains $u_1^2$ and $u_2^2$, and so on. Note that the three corner vertices have degree~2, the vertices on the sides of the grid have degree~4, and the vertices in the middle have degree~6; thus $T_h$ is eulerian.
		
		First, assume that $h=8k+4$. Then, $\frac{|E(G)|}{2}=\frac{3(8k+3)(8k+4)}{4}=48k^2+42k+9$ and thus is odd. Since $T_h$ is eulerian, \Cref{prop-halfEdgesOddAndEulerian} implies that it is not balanceable. The reasoning is the same with $h=8k+5$, with $\frac{|E(G)|}{2}=48k^2+54k+15$.
		
		Now, assume that $h \in \{8k,8k+1\}$. We will prove that there is an independent set $I$ such that $\sum_{v \in I}d(v)=\frac{|E(G)|}{2}$, and apply \Cref{prop-degreeHalfEdges} to complete the proof. We define, for every row except the first, second and last ones, two kinds of independent sets: we call \emph{A-set} of the $i$th row the independent set containing all vertices $u_{2j+1}^i$ for $j \geq 0$; and we call \emph{B-set} of the $i$th row the independent set containing all vertices $u_{2j}^i$ for $j \geq 1$. Note that, if $i$ is odd, then the A-set of the $i$th row contains two vertices of degree~4 and $\frac{i-3}{2}$ vertices of degree~6; and the B-set of the $i$th row contains $\frac{i-1}{2}$ vertices of degree~6. In the following, we will call degree of an A-set (resp. B-set) the sum of the degrees of the vertices it contains.
		\ \\
		
		\noindent\textbf{Case 1:} $h=8k$. Note that in this case, $\frac{|E(G)|}{2}=48k^2-6k$.
		We take the following vertices in $I$:
		\begin{enumerate}
			\item $u_1^1$:
			\item B-sets on rows $3+2i$ for $i \in \{0,\ldots,k-1\}$;
			\item A-sets on rows $3+2k,3+2k+2,\ldots,h-1$.
		\end{enumerate}
		This is depicted on the left-hand side of \Cref{fig-T9}.
		
		Thus, $I$ contains $u_1^1$ which has degree~2, $k$ B-sets which have degree $6(i+1)$ for $i \in \{0,\ldots,k-1\}$, and $3k-1$ A-sets which have degree $8+6i$ for $i \in \{k,\ldots,4k-2\}$. Thus, we have:
		
		\begin{eqnarray*}
			\sum_{v \in I} d(v) & = & 2 + \sum_{i=0}^{k-1} 6(i+1) + \sum_{i=k}^{4k-2} (8+6i) \\
			& = & 2 + \frac{6k(k+1)}{2} + 8(4k-2) + \frac{6(4k-2)(4k-1)}{2} - 8(k-1) - \frac{6(k-1)k}{2} \\
			& = & 48k^2 - 6k
		\end{eqnarray*}
		
		\noindent\textbf{Case 2:} $h=8k+1$. Note that in this case, $\frac{|E(G)|}{2}=48k^2+6k$. If $h=1$ then the result trivially holds since $G$ is the trivial graph. Otherwise,
		we take the following vertices in $I$:
		\begin{enumerate}
			\item $u_1^1$;
			\item A-sets on rows $3,5,\ldots,h-2$, from which we remove $k$ vertices of degree~6 (this is always possible since those A-sets will contain $\sum_{i=0}^{4k-2}i=8k^2-6k+1$ vertices of degree~6, and $k \geq 1 \Rightarrow 8k^2-6k+1>k$);
			\item $u_1^h,u_3^h,\ldots,u_h^h$.
		\end{enumerate}
		This is depicted on the right-hand side of \Cref{fig-T9}.
		
		Thus, $I$ contains $u_1^1$ which has degree~2, $4k-1$ A-sets which have degree $8+6i$ for $i \in \{0,\ldots,4k-2\}$, from which we remove $k$ vertices of degree~6 thus removing $6k$, and the vertices selected on the last row ($4k-1$ of degree~4 and two of degree~2). Thus, we have:
		
		\begin{eqnarray*}
			\sum_{v \in I} d(v) & = & 2 + \sum_{i=0}^{4k-2} (8+6i) - 6k + 4(k-1) + 2+2 \\
			& = & 2 + 8(4k-1) + \frac{6(4k-2)(4k-1)}{2} - 6k + 16k \\
			& = & 48k^2 + 6k
		\end{eqnarray*}
		
		Thus, we have proved that if $h \bmod 8 \in \{0,1\}$, then $T_h$ is balanceable; and that if $h \bmod 8 \in \{4,5\}$, then $T_h$ is not balanceable. This completes the proof of \Cref{thm-triangularGrids}.
	\end{proof}
	
	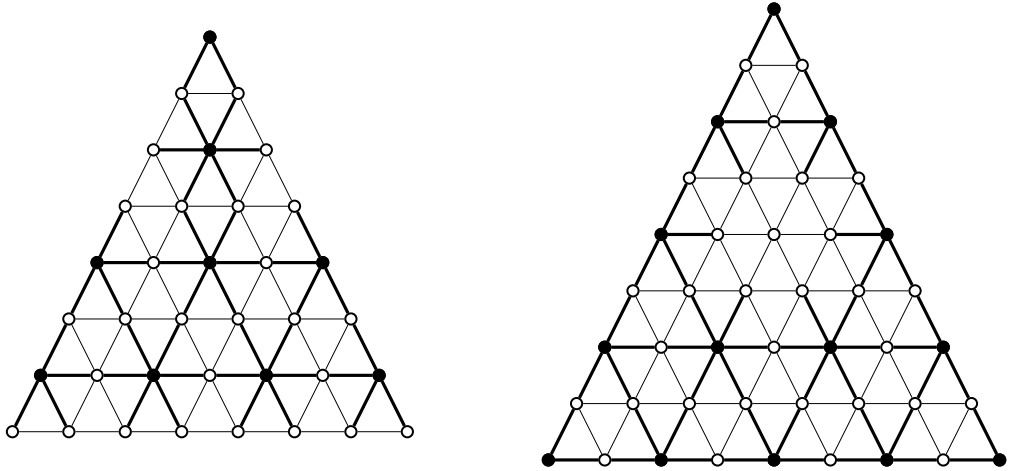
\begin{figure}[h]
		\centering
		\scalebox{0.75}{
			\begin{tikzpicture}
			
			\node (t8) at (0,0) {
				\begin{tikzpicture}
				\foreach \I in {0,...,7} {\node[noeud] (u8\I) at (\I+0.5,1) {};}
				\foreach \I in {0,...,6} {\node[noeud] (u7\I) at (\I+1,2) {};}
				\foreach \I in {0,...,5} {\node[noeud] (u6\I) at (\I+1.5,3) {};}
				\foreach \I in {0,...,4} {\node[noeud] (u5\I) at (\I+2,4) {};}
				\foreach \I in {0,...,3} {\node[noeud] (u4\I) at (\I+2.5,5) {};}
				\foreach \I in {0,...,2} {\node[noeud] (u3\I) at (\I+3,6) {};}
				\foreach \I in {0,...,1} {\node[noeud] (u2\I) at (\I+3.5,7) {};}
				\foreach \I in {0,...,0} {\node[noeud] (u1\I) at (\I+4,8) {};}
				
				\foreach \I/\J in {4/8,4/6,2/4,4/4,6/4,1/2,3/2,5/2,7/2} {
					\draw (\I,\J) node[noeud,fill=black] {};
				}
				
				\foreach[evaluate=\I as \J using {int(\I+1)}] \I in {0,...,6} {\draw (u8\I)--(u8\J);}
				\foreach[evaluate=\I as \J using {int(\I+1)}] \I in {0,...,5} {\draw (u7\I)--(u7\J);}
				\foreach[evaluate=\I as \J using {int(\I+1)}] \I in {0,...,4} {\draw (u6\I)--(u6\J);}
				\foreach[evaluate=\I as \J using {int(\I+1)}] \I in {0,...,3} {\draw (u5\I)--(u5\J);}
				\foreach[evaluate=\I as \J using {int(\I+1)}] \I in {0,...,2} {\draw (u4\I)--(u4\J);}
				\foreach[evaluate=\I as \J using {int(\I+1)}] \I in {0,...,1} {\draw (u3\I)--(u3\J);}
				\draw (u20)--(u21);
				
				\foreach[evaluate=\I as \J using {int(\I+2)}] \I in {0,...,6} {
					\foreach[evaluate=\K as \M using {int(\K-1)}] \K in {\J,...,8}
					{\draw (u\K\I)--(u\M\I);}
				}
				
				\foreach[evaluate=\I as \Q using {int(\I-1)},evaluate=\I as \M using {int(\I+1)}] \I in {1,...,7} {
					\foreach[evaluate=\J as \N using {int(\J+1)}] \J in {0,...,\Q} {
						\draw (u\I\J)--(u\M\N);
					}
				}
				
				\foreach \I/\J in {10/20,10/21,31/20,31/21,31/30,31/32,31/41,31/42,50/40,50/51,50/60,50/61,52/41,52/42,52/51,52/53,52/62,52/63,54/43,54/53,54/64,54/65,70/60,70/71,70/80,70/81,72/61,72/62,72/71,72/73,72/82,72/83,74/63,74/64,74/73,74/75,74/84,74/85,76/65,76/75,76/86,76/87} {\draw[ultra thick] (u\I)--(u\J);}
				\end{tikzpicture}
			};
			
			\node (t9) at (10,0) {
				\begin{tikzpicture}
				\foreach \I in {0,...,8} {\node[noeud] (u9\I) at (\I,0) {};}
				\foreach \I in {0,...,7} {\node[noeud] (u8\I) at (\I+0.5,1) {};}
				\foreach \I in {0,...,6} {\node[noeud] (u7\I) at (\I+1,2) {};}
				\foreach \I in {0,...,5} {\node[noeud] (u6\I) at (\I+1.5,3) {};}
				\foreach \I in {0,...,4} {\node[noeud] (u5\I) at (\I+2,4) {};}
				\foreach \I in {0,...,3} {\node[noeud] (u4\I) at (\I+2.5,5) {};}
				\foreach \I in {0,...,2} {\node[noeud] (u3\I) at (\I+3,6) {};}
				\foreach \I in {0,...,1} {\node[noeud] (u2\I) at (\I+3.5,7) {};}
				\foreach \I in {0,...,0} {\node[noeud] (u1\I) at (\I+4,8) {};}
				
				\foreach \I/\J in {4/8,3/6,5/6,2/4,6/4,1/2,3/2,5/2,7/2,0/0,2/0,4/0,6/0,8/0} {
					\draw (\I,\J) node[noeud,fill=black] {};
				}
				
				\foreach[evaluate=\I as \J using {int(\I+1)}] \I in {0,...,7} {\draw (u9\I)--(u9\J);}
				\foreach[evaluate=\I as \J using {int(\I+1)}] \I in {0,...,6} {\draw (u8\I)--(u8\J);}
				\foreach[evaluate=\I as \J using {int(\I+1)}] \I in {0,...,5} {\draw (u7\I)--(u7\J);}
				\foreach[evaluate=\I as \J using {int(\I+1)}] \I in {0,...,4} {\draw (u6\I)--(u6\J);}
				\foreach[evaluate=\I as \J using {int(\I+1)}] \I in {0,...,3} {\draw (u5\I)--(u5\J);}
				\foreach[evaluate=\I as \J using {int(\I+1)}] \I in {0,...,2} {\draw (u4\I)--(u4\J);}
				\foreach[evaluate=\I as \J using {int(\I+1)}] \I in {0,...,1} {\draw (u3\I)--(u3\J);}
				\draw (u20)--(u21);
				
				\foreach[evaluate=\I as \J using {int(\I+2)}] \I in {0,...,7} {
					\foreach[evaluate=\K as \M using {int(\K-1)}] \K in {\J,...,9}
					{\draw (u\K\I)--(u\M\I);}
				}
				
				\foreach[evaluate=\I as \Q using {int(\I-1)},evaluate=\I as \M using {int(\I+1)}] \I in {1,...,8} {
					\foreach[evaluate=\J as \N using {int(\J+1)}] \J in {0,...,\Q} {
						\draw (u\I\J)--(u\M\N);
					}
				}
				
				\foreach \I/\J in {10/20,10/21,30/20,30/31,30/40,30/41,32/21,32/31,32/42,32/43,50/40,50/51,50/60,50/61,54/43,54/53,54/64,54/65,70/60,70/71,70/80,70/81,72/61,72/62,72/71,72/73,72/82,72/83,74/63,74/64,74/73,74/75,74/84,74/85,76/65,76/75,76/86,76/87,90/80,90/91,92/81,92/82,92/91,92/93,94/83,94/84,94/93,94/95,96/85,96/86,96/95,96/97,98/87,98/97} {\draw[ultra thick] (u\I)--(u\J);}
				\end{tikzpicture}
			};
			
		\end{tikzpicture}
	}
	\caption{The independent set $I$ such that $\sum_{v \in I}d(v)=\frac{|E(G)|}{2}$ for $T_8$ (on the left) and $T_9$ (on the right). Vertices in $I$ are bolded, as well as the out-edges of $I$.}
	\label{fig-T9}
\end{figure}

\section{Conclusion}

In this paper, we extended the study of balanceable graphs initiated in~\cite{CHM19, CHM20}, and used the characterization of balanceable graphs given by \Cref{thm-characterization} to state weaker but more practical conditions for balanceability and non-balanceability. However, the question of how hard is the characterization of \Cref{thm-characterization} remains open. Note that this theorem states that a graph is balanceable if and only if it contains both a cut crossed by half of its edges and an induced subgraph containing half of its edges. In particular, the problem of deciding whether a graph has a cut crossed by exactly half of its edges is a variant on the problem \textsc{Simple-Exact-Cut} (which asks whether an unweighted graph contains a cut crossed by exactly $k$ edges), which is NP-complete:

\begin{prop}
	\label{prop-exactCut}
	Let $G$ be an unweighted graph and $k$ a positive integer. The problem of deciding whether $G$ contains a cut crossed by exactly $k$ edges is NP-complete.
\end{prop}

\begin{proof}
	The problem \textsc{Simple-Exact-Cut} is trivially in NP: a certificate will be a partition $(X,Y)$ and the verifier will simply count the edges with one endpoint in $X$ and the other in $Y$, which can be done in polynomial time.
	
	We use a reduction from \textsc{Simple-Max-Cut}, which asks whether an unweighted graph contains a cut crossed by at least $k$ edges, and was proved to be NP-complete by Garey, Johnson and Stockmeyer~\cite{GJS76}. Let $(G(V,E),k)$ be an instance of \textsc{Simple-Max-Cut}. We let $G'=G \cup K_{1,|E|}$ (that is, the disjoint union of $G$ and a star with $|E|$ leaves) and $k'=k+|E|$. The instance $(G',k')$ of \textsc{Simple-Exact-Cut} is equivalent to $(G,k)$ for \textsc{Simple-Max-Cut}. Indeed, if there is a cut $(X,Y)$ crossed by $\ell$ edges (with $\ell \geq k$) in $G$, then we can take the same cut in the $G$ component of $G'$, and add $k+|E|-\ell$ leaves of $K_{1,|E|}$ to $X$ (which is always possible since $k \leq \ell \leq |E|$) and any remaining leaves as well as the central vertex of the star to $Y$, obtaining a cut crossed by exactly $k+|E|$ edges. Conversely, if there is a cut crossed by exactly $k+|E|$ edges in $G'$, then at most $|E|$ of those edges can be found in the star, so at least $k$ are found in $G$, and we can use the same cut for \textsc{Simple-Max-Cut}.
\end{proof}

However, note that the restriction of \textsc{Simple-Max-Cut} to $k=\frac{|E|}{2}$ has a trivial answer since there is a polynomial algorithm giving a cut crossed by at least half the edges of a graph~\cite{KRY07}. Nonetheless, it seems that guaranteeing a cut crossed by exactly half of the edges would remain difficult. This is one of the two angles to tackle the problem of the computational complexity of balanceability by using the characterization given by \Cref{thm-characterization}, the other one being the existence of an induced subgraph containing exactly half the edges of the graph. Since both of those problems seem to be difficult on their own, and that there does not seem to be an easy way to conjugate them, we conjecture that the problem of balanceability is NP-complete.

\begin{conjecture}
	The problem of deciding whether a given graph is balanceable is NP-complete.
\end{conjecture}

In this paper, we also studied the balanceability of a special class of circulant graphs as well as of the rectangular and triangular grids with an even number of edges, for which we fully characterized those that are balanceable and those that are not. An interesting open problem would be to further the classification of the circulant graphs between balanceable and non-balanceable. Other graph classes that could be interesting to study are cubic graphs, as well as not necessarily regular classes such as $k$-trees, outerplanar, or even planar graphs. 

\section*{Acknowledgements}

We would like to thank the anonymous referees for their excellent advice, as well as suggesting \Cref{prop-regularBipartite4n}. We also would like to thank BIRS-CMO for hosting the workshop Zero-Sum Ramsey Theory: Graphs, Sequences and More (19w5132),  where many fruitful discussions arose that contributed to a better understanding of these topics.

\end{document}